\documentclass[12pt]{article}
\usepackage{amsmath,amssymb}
\usepackage{amsthm}

\def\FC{\mathcal{F}}

\def\E{\mathbf{E}}
\def\N{\mathbf{N}}

\def\P{\mathbf{P}}
\def\R{\mathbf{R}}
\def\Z{\mathbf{Z}}

\def\1{\mathbf{1}}

\def\tr{\rm{tr}}

\def\al{\alpha}
\def\be{\beta}
\def\pa{\partial}
\def\ep{\epsilon}
\def\de{\delta}

\newcommand{\si}{\sigma}

\newcommand{\De}{\Delta}

\newcommand{\Si}{\Sigma}
\newcommand{\Om}{\Omega}
\newcommand{\La}{\Lambda}

\newtheorem{prop}{Proposition}[section]
\newtheorem{theorem}{Theorem}[section]
\newtheorem{lemma}{Lemma}[section]

\newtheorem{remark}{Remark}

\begin{document}
\title{CTRW approximations for fractional equations with variable order}

\author{V. N. Kolokoltsov\footnote{Department of Statistics, University of Warwick, Coventry CV4 7AL UK,}
$^,$\footnote{ National Research University Higher School of Economics, Moscow}
\footnote{Email: v.kolokoltsov@warwick.ac.uk and v.n.kolokoltsov@gmail.com}
}

\maketitle

\begin{abstract}
The standard diffusion processes are known to be obtained as the limits of appropriate
random walks. These prelimiting random walks can be quite different however.
The diffusion coefficient can be made responsible for the size of jumps
or for the intensity of jumps. The "rough" diffusion limit does not feel the difference.
The situation changes, if we model jump-type approximations via CTRW with non-exponential waiting times.
If we make the diffusion coefficient responsible for the size of jumps and take
waiting times from the domain of attraction of an $\alpha$-stable law with a constant intensity $\al$,
then the standard scaling would lead in the limit
of small jumps and large intensities to the most standard fractional diffusion equation.
However, if we choose the CTRW approximations with fixed jump sizes and use the diffusion coefficient
to distinguish intensities at different points, then we obtain in
the limit the equations with variable position-dependent fractional derivatives.
In this paper we build rigorously these approximations and prove their convergence
to the corresponding fractional equations for the cases of multidimensional diffusions
and more general Feller processes.
\end{abstract}


\section{Introduction}

The standard diffusion equation
\begin{equation}
\label{eqdifeq}
\frac{\pa u}{\pa t}(t,x) =\frac12 a(x) \frac{\pa ^2 u}{\pa x^2}(t,x),
\end{equation}
with a positive continuous function $a(x)$,
and the corresponding diffusion process are known to be obtained as the limit of appropriate
random walks. These prelimiting random walks can be quite different however. For example,
approximating random walks in continuous time can be chosen as jump type processes with the generators
\[
L_hf(x)=\frac{1}{2h^2}[f(x+ h\sqrt{a(x)})+f(x- h\sqrt{a(x)})-2f(x)],
\]
or as jump type processes with the generators
\[
\tilde L_hf(x)=\frac{a(x)}{2h^2}[f(x+ h)+f(x- h)-2f(x)],
\]
since both $L_h$ and $\tilde L_h$ tend to $(1/2)a(x)(d^2/dx^2)$, as $h\to 0$.
In the first approximation, the diffusion coefficient $a(x)$ is responsible for the size of jumps (with constant intensity),
and in the second approximation, it is responsible for the intensity of jumps (with constant sizes).
The "rough" diffusion limit does not feel the difference.
The situation changes, if we model jump-type approximations via CTRW with non-exponential waiting times.
If we make the diffusion coefficient responsible for the size of jumps and take
waiting times from the domain of attraction of an $\alpha$-stable law with a constant intensity $\al$,
then the standard scaling would lead (in the limit
of small jumps and large intensities) to the most standard fractional diffusion equation

\begin{equation}
\label{eqdifeqfr}
D^{\al}_{0+*} u(t,x) =\frac12 a(x) \frac{\pa ^2 u}{\pa x^2}(t,x),
\end{equation}
where $D^{\al}_{0+*}$ is the so-called Caputo-Dzerbashyan fractional derivative.

However, if we will use the CTRW approximations with fixed jump sizes and use $a(x)$
to distinguish intensities at different points, then (as will be shown) we get in
the limit the equation with a variable position-dependent fractional derivative:
\begin{equation}
\label{eqdifeqfr1}
D^{\al a(x)}_{0+*} u(t,x) =\frac12 \frac{\pa ^2 u}{\pa x^2}(t,x).
\end{equation}

Of course, for any decomposition $a(x)=b(x)c(x)$ with positive $b(x),c(x)$, one can use
$b(x)$ as a function controlling the intensity of jumps and $c(x)$ as a function controlling
the spread of the jumps. In this scenario we get in the limit the equation
\begin{equation}
\label{eqdifeqfr2}
D^{\al b(x)}_{0+*} u(t,x) =\frac12 c(x) \frac{\pa ^2 u}{\pa x^2}(t,x).
\end{equation}

From this point of view, all equations \eqref{eqdifeqfr2} are equally
legitimate fractional extensions of the diffusion equation \eqref{eqdifeq}.

Similarly (though technically more cumbersome), for a time dependent diffusion equation
\begin{equation}
\label{eqdifeqnonhom}
\frac{\pa u}{\pa t}(t,x) =\frac12 a(t,x) \frac{\pa ^2 u}{\pa x^2}(t,x),
\end{equation}
factorising $a(t,x)=b(t,x)c(t,x)$ we get by the CTRW approximation the fractional equation with variable order,
which is time- and position-dependent:
\begin{equation}
\label{eqdifeqnonhomfr}
D^{\al b(t,x)}_{0+*} u(t,x) =\frac12 c(t,x) \frac{\pa ^2 u}{\pa x^2}(t,x).
\end{equation}

In this paper we aim to develop these CTRW approximations in a general multidimensional case.
Namely, let us look at the CTRW approximations and possible corresponding fractional limits
for the diffusion processes generated by the operator
\begin{equation}
\label{eqdifoper}
L_{dif}f(x) =\frac12 {\tr} \left(G(x) \frac{\pa ^2 f}{\pa x^2}(x)\right)
=\frac12 \sum_{ij} G_{ij}(x) \frac{\pa ^2 f}{\pa x_i\pa x_j}(x), \quad x\in \R^d,
\end{equation}
with $G(x)$ a positive $d\times d$-matrix,
and stable-like processes generated by the operator
\begin{equation}
\label{eqstaoper}
L_{st}^{\be}f(x) =\int_{S^{d-1}}|(s, \nabla f(x)|^{\be} \mu (x,s)\, ds, \quad x\in \R^d,
\end{equation}
with $\be\in (0,2)$, $\mu (x,s)$ an even in $s$ positive function on $\R^d\times S^{d-1}$ and
$ds$ Lebesgue measure on the unit sphere $S^{d-1}$.   In particular, if $\mu(x,s)=1$
(that is, the spectral measure $\mu(x,s) \, ds$ is uniform), then
$L_{st}^{\be}=\si |\De|^{\be/2}$ (with $\si$ a constant depending on $d$) becomes a standard fractional Laplacian.
To avoid unnecessary repetitions, we shall speak about $L_{st}^{\be}$ with $\be\in (0,2]$ assuming that for
$\be=2$ the diffusion operator \eqref{eqdifoper} is meant.

We shall look at the CTRW approximations arising from the equations
\begin{equation}
\label{eqstaeq}
\frac{\pa u}{\pa t}(t,x) =a(t,x) L_{st}^{\be} u(t,x),
\end{equation}
with some positive function $a(x)$. We stress that the separation of a multiplier $a(t,x)$
for a diffusion or a stable generator is not intrinsic, it can be done in arbitrary way. Once the choice
is made, we shall analyse the jump-type approximations to the processes governed by $a(t,x) L_{st}^{\be}$,
where $a(t,x)$ stands for the intensity of jumps and
$L_{st}^{\be}$ stands for the distribution of jumps.
We shall show that the limiting processes are governed by the equations of type
\[
D^{\al a(t,x)}_{0+*} u(t,x) =L_{st}^{\be}u(t,x),
\]
extending \eqref{eqdifeqnonhomfr}.

Our results are extendable to equations
\begin{equation}
\label{eqstaeqgen}
\frac{\pa u}{\pa t}(t,x) =a(t,x) L u(t,x),
\end{equation}
with a generator of a general Feller process. We shall mostly work with \eqref{eqstaeq},
where concrete assumptions on coefficients are easy to check, and then comment on regularity
assumptions required for a general $L$.

The CTRW models summing independent terms were analysed mathematically in many papers,
see e. g. \cite{Meer09},
following their development in physics literature, see e.g. \cite{Zas94}.
The CTRW approximation for general models of Markov processes subordinated
by their monotone coordinate turning in the scaling limit to a generalised fractional equation
(of variable order) were seemingly first discussed mathematically in \cite{Ko09}.
There is now a heavy body of literature on the variable order fractional equations.
We are not trying to review it here, but rather refer to extensive reviews
\cite{Ding21} and \cite{Sun} dealing with both mathematical and applied questions.
Let us mention specifically paper \cite{Straka}, where CTRW approximations
to equation \eqref{eqdifeqfr1} were discussed in detail, and paper \cite{SavToa}, with careful derivation of
well-posedness for rather general variable order fractional equations.
As basic references to general rules of fractional calculus and its applications we mention
books \cite{Kir94}, \cite{Podlub99}, \cite{Scalabook}.

The rest of the paper is organised as follows. In Section \ref{secCTRW}
we introduce carefully various CTRW approximations for diffusion or stable-like processes, when
their coefficients can be used either to distinguish the sizes of jumps or
their intensities. In Section \ref{secres} our main results are formulated devoted to
the convergence of these approximations to processes solving fractional equations
of variable order, and to the explicit representations of the solutions of these equations.
The main result is actually Theorem \ref{th2} on the CTRW approximations and specific representation
for their limit. The proof is based on the ideas from \cite{Ko09}, where they were presented
however rather sketchy. Theorem \ref{th3} is essentially known.
Its detailed proof is given in \cite{SavToa}. We present here a different proof adding two details.
Namely, our equation is a bit more general (variable fractional order depends on both space and time), and
moreover we prove well-posedness of classical solutions (unlike only the mild solutions of \cite{SavToa}).
In Section \ref{secrepsub} we derive a handy general representation for the distributions
of Markov processes subordinated by its monotone coordinate. It supplies an important
ingredient to the proof of our main results, but may be of independent interest.
Section \ref{secrates} presents a well known result on the convergence of
the standard CTRWs, but enhanced by the exact rates of (weak) convergence.
Sections \ref{secpr1} - \ref{secpr3} contain proofs of our main results.

Notice finally that similar approach to interacting particle systems leading
to the new variable order fractional kinetic equations was developed in
\cite{KolMal} and \cite{KoTro}. Quite different variable order kinetic equations
are analysed in \cite{Fed}.

We shall use standard notations for the spaces of smooth function: $C(\R^d)$
the Banach space of continuous functions equipped with sup-norm,
$C^k(\R^d)$, $k\in \N$, the Banach space of
$k$-times continuously differentiable functions with bounded derivative of
order up to $k$,  $C_{\infty}(\R^d)$ the closed subspace of $C(\R^d)$
consisting of functions vanishing at infinity, $C^k_{\infty}(\R^d)$ the closed subspace
of $C^k(\R^d)$ consisting of functions such that all its derivatives up to order $k$
belong to $C(\R^d)$.

\section{CTRW approximation}
\label{secCTRW}

We are looking for the CTRW approximations to the equations \eqref{eqstaeq}.

For the operator $L_{st}^{\be}$, we choose standard random walk approximations
\begin{equation}
\label{eqstaoperap}
L_{st,\tau}^{\be}f(x)=\frac{1}{\tau}\int (f(x+\tau^{1/\be}z)-f(x)) p(x, dz),
\end{equation}
with $p(x,dz)$ a symmetric probability kernel in $\R^d$
(i.e. $\int f(z) p(x,dz)=\int f(-z) p(x,dz)$), so that
\begin{equation}
\label{eqstaoperap1}
\sup_x |(L_{st}^{\be}-L_{st,\tau}^{\be})f(x)| \to 0, \quad \text{as} \quad \tau\to 0,
\end{equation}
uniformly for $f$ from any bounded subset of $C^2(\R^d)$.

The operator  $L_{st,\tau}^{\be}$ generates a jump-type Markov process with the intensity $\tau^{-1}$
and the distribution of jumps given by $p(x, dz)$.

\begin{remark}
Of course, such approximations are not unique. The existence is however more or less obvious. For the diffusion operator
\eqref{eqdifoper} the measures $p(x,dz)$ can be (and usually are) taken as discrete, for instance,
of the type
 \[
 \sum_i(\de_{a_ie_i}+\de_{-a_ie_i})+\sum_{i>j}(\de_{b_{ij}(e_i+e_j)}+\de_{-b_{ij}(e_i+e_j)}),
 \]
with some numbers $a_i, b_{ij}$, where $\{e_i\}$ is the standard basis in $\R^d$. For the stable generator
one can take as $p(x,dz)$ a probability measure with an appropriate power tail,
see e.g. \cite{Meer09} or Section 8.3 from \cite{Ko11}.
\end{remark}

For a generator $L$ of an arbitrary Feller process with a core $C^k_{\infty}(\R^d)$, where $k$ equals one or two,
on can use an arbitrary approximation $L_{\tau}$ of the type
\begin{equation}
\label{eqgenoperap}
L_{\tau} f(x)=\frac{1}{\tau}\int (f(x+z)-f(x)) p_{\tau}(x, dz),
\end{equation}
with $p_{\tau}(x,dz)$ a family of probability kernels in $\R^d$ such that
\begin{equation}
\label{eqgenoperap1}
\sup_x |(L-L_{\tau})f(x)| \to 0, \quad \text{as} \quad \tau\to 0,
\end{equation}
uniformly for $f$ from any bounded subset of $C^k_{\infty}(\R^d)$.

We shall now construct a process with waiting times having power tails. For exponential waiting times,
the intensity $a(s,x)$ governs the tail distribution $\P(T>t)=e^{-ta(s,x)\tau}$ with a scaling parameter $\tau$.
Hence, for the power tail the analogous dependence will be the power law
\begin{equation}
\label{eqpowertail}
\P(T>t)\sim \frac{1}{a(s,x)\al} t^{-a(s,x)\al }, \quad t\to \infty.
\end{equation}

To simplify presentation we shall make \eqref{eqpowertail} more precise. Namely, we assume
that these distributions have continuous densities $Q_{s,x}(r)$ such that
\begin{equation}
\label{eqpowertail1}
Q_{s,x}(r)=r^{-1-a(s,x) \al} \,\, \text{for} \,\, r\ge B, \,\, \text{and} \,\, Q_{s,x}(r) \le 1 \,\, \text{for all} \,\, r,
\end{equation}
with some $B>0$ uniformly for all $s,x$.

\begin{remark}
(i) The coefficients at the power of $t$ in \eqref{eqpowertail} are chosen for convenience.
If not used here, then it would appear in \eqref{eqpowertail1} and
in the limiting equation. (ii) Simplification arising from \eqref{eqpowertail1}
is not essential and can be dispensed with.
\end{remark}

As usual for CTRW modelling, we start by building an auxiliary scaled enhanced Markov chain
with spatial jumps distributed according to $p(x,dz)$ and with an additional
coordinate $r$ specifying the total waiting times. In the usual CTRW scaling
(see e.g. \cite{Meer09} and \cite{Ko11}) one scales discrete times by $\tau$
and the waiting times by the multiplier $\tau^{1/(\al a(s,x))}$ (see also \eqref{eq10apprrateCTRW}).  Therefore,
we define the Markov chain $(X^{\tau}_{x,s}, S^{\tau}_{x,s})(k\tau)$ evolving in discrete times $\tau k$,
$k\in \N$, such that $(X^{\tau}_{x,s}, S^{\tau}_{x,s})(0)=(x,s)$ and,
if a state of the chain in some discrete
time $k\tau$ is $(q,v)$, then in the next moment $(k+1)\tau$ it will turn to the state
\[
(q+\tau^{1/\be}y, v+\tau^{1/(\al a(v,q))}r),
\]
with $y$ distributed according to $p(x,dy)$ and $r$ distributed according to $Q_{v,q}$.
We obtain the Markov chain with transition operators in time $\tau$ given by the formula
\begin{equation}
\label{eqtransitionenh}
U^{\tau}F(x,s)= \int_{\R_+}\int_{\R^d} F(x+\tau^{1/\be} y,  s+\tau^{1/(\al  a(s,x))} r) Q_{s,x}(r) \, dr \,  p(x,dy).
\end{equation}

When working with an arbitrary generator $L$ (rather than more concrete $L^{\be}_{st}$) the corresponding
chain is defined by the equation
\begin{equation}
\label{eqtransitionenhgen}
U^{\tau}F(x,s)= \int_{\R_+}\int_{\R^d} F(x+y,  s+\tau^{1/(\al  a(s,x))} r) Q_{s,x}(r) \, dr \,  p_{\tau}(x,dy).
\end{equation}

We are interested in the value of the first coordinate $X^{\tau}_{x,s}$
evaluated at the random time $k\tau$ such that the total waiting time $S^{\tau}_{x,s}(k\tau)$
reaches $t$, that is, at the time
\[
k\tau =T_{x,s}^{\tau}(t)=\inf \{ m\tau: S^{\tau}_{x,s} (m\tau) \ge t\},
\]
so that $T_{x,s}^{\tau}$ is the inverse process to $S^{\tau}_{x,s}$.
We define the {\it scaled CTRW approximation for the process with the intensity of jumps governed by $a(x)$
and the distribution of jumps governed by $p(x,dy)$} as the (non-Markovian) process
 \begin{equation}
\label{scaledmfchainwithtail}
\tilde X^{\tau}_{x,s}(t)= X^{\tau}_{x,s}(T_{x,s}^{\tau}(t)).
\end{equation}

We are going to identify the weak limit of the process $\tilde X^{\tau}_{x,s}(t)$, as $\tau\to 0$,
and to show that the limiting processes are governed by the equations extending
\eqref{eqdifeqnonhomfr}.
As an auxiliary step, we shall identify the weak limit of the Markov chain $(X^{\tau}_{x,s}, S^{\tau}_{x,s})$,
that is, the weak limit of the chains with transitions
 $[U^{\tau}]^{[t/\tau]}$, where $[t/\tau]$ denotes the integer
 part of the number $t/\tau$, as $\tau\to 0$.
It is known (see e.g. Theorem 19.28 of \cite{Kal} or Theorem 8.1.1 of \cite{Ko11})
that if such chain converges to a Feller process, then the generator of this limiting process
can be obtained as the limit
 \begin{equation}
\label{scalingMchains}
\La F= \lim_{\tau\to 0} \frac{1}{\tau}(U^{\tau}F-F).
\end{equation}

\section{Main results}
\label{secres}

Let us assume the strict non-degeneracy condition:
\begin{equation}
\label{eqcondnondeg}
g_1 \le G(x) \le g_2, \,\, \text{or} \,\,  m_1\le m(x,s)\le m_2
\end{equation}
for the spatial part and
\begin{equation}
\label{eqcondnondeg1}
a_1\le a(t,x) \le a_2
\end{equation}
for the temporal part,
with some positive constants $g_i, m_i, a_i$ such that $a_2 \al <1$,
where the first inequality \eqref{eqcondnondeg} is understood in the sense of matrix,
and the smoothness condition:
\begin{equation}
\label{eqcondsmoo}
a(.), G(.), m(.,s)\in C^4(\R^d),
\end{equation}
with $m(.,s)$ having a bounded norm uniformly in $s$. Let us assume
\eqref{eqstaoperap1} for the approximating jumps and \eqref{eqpowertail1}
for the distributions of waiting times.

\begin{theorem}
\label{th1}
(i) For $F\in C^2_{\infty}(R^{d+1})$ the limit in \eqref{scalingMchains} exists and
\[
\La F(x,s)= \lim_{\tau\to 0} \frac{1}{\tau}(U^{\tau}F-F)(x,s)
\]
\begin{equation}
\label{eq1th1}
=\int_0^{\infty} \frac {F(x, s+r)-F(x,s)}{r^{1+\al a(s,x)}} dr
+L_{st}^{\be}F(x,s).
\end{equation}

(ii) The operator $\La$ defined by the r.h.s. of \eqref{eq1th1} generates
a (conservative time-homogeneous) Feller process $(X_{x,s}, S_{x,s})(t)$
in $\R^{d+1}$ (with the initial condition $(X_{x,s}, S_{x,s})(0)=(x,s)$)
with continuous (in fact, even smooth) transition probability density
 $G(t;x,s; y,v)$, $t>0$,
and the corresponding  Feller semigroup in $C_{\infty}(\R^{d+1})$
such that the space $C^2_{\infty}(\R^{d+1})$ is its invariant core. The process
$(X_{x,s}, S_{x,s})(t)$ belongs to the class of processes usually referred to as stable-like processes.

(iii) The chains with transitions $[U^{\tau}]^{[t/\tau]}$, with $U^{\tau}$
given by \eqref{eqtransitionenh}, converge weakly to the process  $(X_{x,s}, S_{x,s})(t)$,
and the operators $[U^{\tau}]^{[t/\tau]}$ converge to the transition operators of the process
$(X_{x,s}, S_{x,s})(t)$ strongly and uniformly for compact intervals of time.
\end{theorem}

 \begin{theorem}
\label{th2}
(i) The marginal distributions of the scaled CTRW  \eqref{scaledmfchainwithtail} converge
 to the marginal distributions of the process
 \begin{equation}
\label{eq1th2}
\tilde X_{x,s}(t)= X_{x,s}(T_{x,s}(t)),
\end{equation}
where
\[
T_{x,s}(t)=\inf \{r: S_{x,s} (r) \ge t\},
\]
that is, for a bounded continuous function $F(x)$, it holds that
 \begin{equation}
\label{eq1ath2}
\lim_{\tau \to 0}  \E F(\tilde X^{\tau}_{x,s}(t))= \E F(\tilde X_{x,s}(t)).
\end{equation}

(ii) For any $F\in C_{\infty}(\R^d)$ and arbitrary $K>0$,
 \[
\E [F(\tilde X_{x,s}(t))\1(T_{x,s}(t)\in [1/K,K])]
\]
\begin{equation}
 \label{eq20th2}
 =\int dy \int_{1/K}^K du \int_s^t dv \frac{(t-v)^{-a(v,y)\al}}{a(v,y)\al} \, G(u;x,s; y,v) F(y),
\end{equation}
where $\theta_{\ge t}$ is the indicator function of the interval $[t,\infty)$.

This formula extends
to infinite $K$, that is
\begin{equation}
 \label{eq21th2}
\E [F(\tilde X_{x,s}(t))]
 =\int dy \int_0^{\infty} du \int_s^t dv \frac{(t-v)^{-a(v,y)\al}}{a(v,y)\al} \, G(u;x,s; y,v) F(y).
\end{equation}
\end{theorem}

\begin{theorem}
\label{th3}
For any $F\in C^2(\R^d)$, the evolution of averages $F(x,s)=\E F(\tilde X_{x,s}(t))$
represents the unique solution to the mixed fractional differential equation
 \begin{equation}
\label{eq2th2}
D_{t-*}^{\al a(s,x)}F(x,s)= L_{st}^{\be} F(x,s),
\quad s \in [0,t],
\end{equation}
with the terminal condition $F(x,t)=F(x)$,
where the right fractional derivative (of Caputo-Dzerbashyan type)
acting on the variable $s\le t$ of $F(x,s)$ is defined as
 \begin{equation}
\label{eq3th2}
D_{t-*}^{\al a(s,x)}g(s) =- \int_0^{t-s} \frac {g(s+r)-g(s)}{r^{1+\al a(s,x)}} dy
- (g(t)-g(s)) \frac{(t-s)^{-\al a(s,x)}}{a(s,x) \al}.
\end{equation}
\end{theorem}

\begin{remark}
The limiting equation \eqref{eq3th2} is written in terms of the right fractional derivative,
because we have chosen the forward motion for the auxiliary coordinate $S_{s,x}$ of the
Markov process $(X_{x,s}, S_{x,s})(t)$. The analogous equations with left derivatives,
like in \eqref{eqdifeqfr2}, can be obtained by changing the direction of $S_{s,x}$.
\end{remark}

\begin{remark}
For the case of a general operator $L$ the results above remain true (with $L$ instead of $L_{st}^{\be}$)
if we assume that the process generated by $\La$ from \eqref{eq1th1} has a continuous transition
probability density, which is sufficiently smooth so that smooth functions form a core for its
Feller semigroup, and which has integrability condition for small and large times ensuring convergence
of the integral in \eqref{eq20th2}.
\end{remark}

\section{A general representation for a subordination}
\label{secrepsub}

In this section we derive a representation for Markov processes subordinated by its monotone coordinate.
It supplies the proof for the part (ii) of Theorem \ref{th1}, but may be of independent interest.

\begin{lemma}
\label{lemsubs}
Let $Y_s$ be an adapted process on a stochastic basis $(\Om, \FC, \FC_t, P)$ and $\Si$ a stopping time.
Let the pairs $(Y_s, \Si)$ have joint densities $g_s(y, \si)$ for $s>0$ such that, for $s$ from any bounded interval
separated from zero, $g_s(y,\si)$ is bounded and right continuous in $s$.
Then, for any $K>0$ and a continuous bounded function $F(y)$,
\begin{equation}
 \label{eq1lemsubs}
\E [F(Y_{\Si})\1(\Si\in [1/K,K])]= \int dy \int_{1/K}^K d\si \, g_{\si}(y,\si) F(y).
\end{equation}
\end{lemma}

 \begin{proof}
Let us choose the discrete approximations
$\Si_{\tau}=\tau [\Si/\tau+1]$ (square brackets denoting the integer part), so that $\Si_{\tau} =k\tau$
when $(k-1)\tau \le \Si< k\tau$. Thus $\Si_{\tau}$ are right continuous stopping times such that
$\Si_{\tau}\ge \Si$ for all $\tau$, they depend monotonically on $\tau$ and $\Si_{\tau}\to \Si$, as $\tau \to 0$.
Then we can write
\[
\E [F(Y_{\Si})\1(\Si\in [1/K,K])]=\lim_{\tau \to 0} \E [F(Y_{\Si_{\tau}})\1(\Si \in [1/K,K])]
\]
\[
=\lim_{\tau \to 0} \sum_{k=1}^{\infty} \E [F(Y_{\Si_{\tau}})\1(\Si \in [1/K,K]) \1(\Si_{\tau}=k\tau)]
\]
\[
=\lim_{\tau \to 0} \sum_{k=1}^{\infty} \E \left[F(Y_{k\tau})\1(\Si \in [1/K,K]) \1(\Si \in [(k-1)\tau, k\tau))\right]
\]
\[
= \lim_{\tau \to 0} \int dy \int_{1/K}^K d\si \, g_{\si_{\tau}}(y,\si) F(y).
\]
And therefore
\[
\E [F(Y_{\Si})\1(\Si\in [1/K,K])]
= \int dy \int_{1/K}^K d\si \, g_{\si}(y,\si) F(y)
\]
\[
+\lim_{\tau \to 0} \int dy \int_{1/K}^K d\si \, (g_{\si_{\tau}}-g_{\si})(y,\si) F(y).
\]
Since the last term tends to zero (by the right continuity of $g$ and dominated convergence),
we obtain \eqref{eq1lemsubs}.

\end{proof}

\begin{remark}
(i) To get the same without the restrictions $[1/K,K]$
one needs some integrability conditions for $g_s$ for small and large $s$.
(ii) The existence of the density $g_s(y,\si)$ can be weaken to the condition
\[
\P(Y_s\in dy, \Si\in d\si)=g_s(\si, dy) d\si
\]
such that $g_s(\si, dy)$ weakly right continuous in $s$ uniformly in $\si$ (on bounded intervals).
\end{remark}

\begin{lemma}
\label{lemsubs1}
Let $(Y_s, V_s)$, $Y\in \R^d$, $V \in \R$ be a Markov process with the generator $L+A$, where $L$ acts on the first variable
and does not depend on the second one (so that $Y_s$ is itself a Markov process),
and $A$ acts on the second variable, but may have coefficients depending
on the first variable.   Let the process $(Y_s, V_s)$ have a continuous (for $s>0$)
transition density $G(s;y_0,v_0;y,v)$
and let $G_Y(s,y_0,y)$ be the transition density for the process $Y$ so that
\[
G_Y(u,y_0,y)=\int G(u;y_0,v_0;y,v) dv
\]
for any $v_0$.

Then the random vector $(Y_s,V_u)$ has a density $\phi_{y_0,v_0}(s,u;y,v)$ such that, for $s> u>0$,
\begin{equation}
 \label{eq1lemsubs1}
\phi_{y_0,v_0}(s,u;y,v)=\int G_Y(s-u,p, y) G(u;y_0,v_0;p,v) \, dp.
\end{equation}
Moreover,
\begin{equation}
 \label{eq2lemsubs1}
\frac{\pa}{\pa u} \phi_{y_0,v_0}(s,u;y,v)=\int G_Y(s-u,p, y) A^*_{(v;p)}G(u;y_0,v_0;p,v) \, dp,
\end{equation}
where $A^*$ denotes the adjoint operator to $A$ and the lower index $(v;p)$ indicates that this operator
acts on the variable $v$, while the value of the spacial variable is $p$.

\end{lemma}

\begin{remark}
Generally speaking, formula \eqref{eq2lemsubs1} may hold only
in the sense of generalized functions. To ensure that it holds as an equation for functions,
some additional mild regularity for transitions must be assumed.
\end{remark}

\begin{proof}
For a bounded continuous function $F$, we have
\[
\E F(Y_s,V_u)=\int \E (F(Y_s,V_u)|Y_u=p,V_u=v)  G(u;y_0,v_0; p,v) \, dp dv
\]
\[
=\int \E (F(Y_s,v)|Y_u=p)  G(u;y_0,v_0; p,v) \, dp dv
=\int F(y,v) G_Y(s-u,p,y) G(u;y_0,v_0; p,v) \, dy dp dv,
\]
implying \eqref{eq1lemsubs1}.

Next,
\[
\frac{\pa}{\pa u} \phi_{y_0,v_0}(s,u;y,v)=-\int (L_{(p)}G_Y(s-u,p,y)) G(u;y_0,v_0; p,v) \, dp
\]
\[
+\int G_Y(s-u,p,y)) (L^*_{(p)}+A^*_{(v;p)})G(u;y_0,v_0;p,v) \, dp.
\]
The key observation is that the terms with $L_{(p)}$ cancel implying  \eqref{eq2lemsubs1}.
\end{proof}

\begin{lemma}
\label{lemsubs2}
Under the assumption of Lemma \ref{lemsubs1}, let
\begin{equation}
 \label{eq1lemsubs2}
A \phi(v) =\int_0^{\infty} ( \phi (v+w)-\phi(v))\nu(v,w;p) \, dw,
\end{equation}
where $p$ is the value of the spatial variable,
where $\nu$ is a strictly positive function for $w>0$ such that
\[
\sup_{v,p} \int \min(1,w) \nu(v,w;p) \, dw <\infty.
\]

Then the coordinate $V_t$ is a strictly increasing process a.s., so that its generalized inverse
(or hitting times) process
\[
Z_t=\sup\{ u\ge 0: V(u)\le t\}=\inf\{ u\ge 0: V(u)> t\}, \quad t>v_0,
\]
is continuous; and for the density $g_{y_0,v_0}(s,t;y,z)$ of the pair $(Y_s,Z_t)$, $s> 0$, $t> v_0$,
the following formula holds:
\[
g_{y_0,v_0}(s,t;y,z)=  \frac{\pa}{\pa z} \int_t^{\infty} \phi_{y_0,v_0}(s,z;y,v_0) \, dv
\]
\begin{equation}
 \label{eq2lemsubs2}
 =\int_t^{\infty}dv \int_{\R^d} dp \,  G_Y(s-z,p,y) A^*_{(v;p)}G(z;y_0,v_0; p,v).
\end{equation}
In particular,
\begin{equation}
 \label{eq3lemsubs3}
\lim_{z\to s}g_{y_0,v_0}(s,t;y,z)
= \int_t^{\infty} A^*_{(v;y)}G(s;y_0,v_0;y,v) \, dv.
\end{equation}

\end{lemma}

%\begin{remark} We write here and in what follows $A_{(v;p)}$ for clarity,
%to indicate that $A$ acts on the variable $v$
%with the value of the spatial variable being taken as $p$.
%\end{remark}

\begin{proof}
The fact that $V_t$ is strictly increasing follows from \eqref{eq1lemsubs2},
which in turn implies the continuity of $Z_t$.

Consequently, since $(Z_t\le z)=(V_z\ge t)=(V_z>t)$ a.s., it follows that
\[
g_{y_0,v_0}(s,t;y,z)=\frac{\pa^2}{\pa y \pa z} \P(Y_s\le y, V_z \ge t)
\]
\[
=\frac{\pa^2}{\pa z \pa y} \int_{-\infty}^y \int_t^{\infty} \phi_{y_0,v_0}(s,z;\tilde y,v) \, dv dy
=\frac{\pa}{\pa z} \int_t^{\infty} \phi_{y_0,v_0}(s,z;y,v) \, dv.
\]
Substituting \eqref{eq2lemsubs1}, we get \eqref{eq2lemsubs2}.
\end{proof}

\begin{theorem}
\label{thsubs}
Under the assumptions of Lemma \ref{lemsubs1}
\begin{equation}
 \label{eq1thsubs}
\E [F(Y_{Z_t})\1(\Z_t\in [1/K,K])]
= \int dy \int_{1/K}^K ds \int_t^{\infty} dv \, A^*_{(v;y)}G(s;y_0,v_0; y,v) F(y),
\end{equation}
and also
\[
\E [F(Y_{Z_t})\1(\Z_t\in [1/K,K])]
= \int dy \int_{1/K}^K ds \int_0^t dv \, (A_{(v;y)}\theta_{\ge t})(v)G(s;y_0,v_0; y,v) F(y)
\]
\begin{equation}
 \label{eq2thsubs}
 =\int dy \int_{1/K}^K ds \int_{v_0}^t dv \int_{t-v}^{\infty} \nu(v,w;y) \, dw \, G(s;y_0,v_0; y,v) F(y),
\end{equation}
where $\theta_{\ge t}$ is the indicator function of the interval $[t,\infty)$.
\end{theorem}

\begin{remark}
(i) Formula \eqref{eq1thsubs} in this general form was seemingly first derived in \cite{Ko09}
(though the proof was a bit sketchy there) and was reproduced in monograph \cite{Ko11}.
The advantage of its modification \eqref{eq2thsubs} is that for its validity only the continuity
of the density $G(s;y_0,v_0; y,v)$ is required (and even this condition can be weaken), and not any smoothness
that may be needed to make sense out of $A^*_{(v;y)}G(s;y_0,v_0; y,v)$.

(ii)  It follows that the density of the random variable $Y_{Z_t}$ equals
\begin{equation}
 \label{eq3thsubs}
 \int_0^{\infty} ds \int_{v_0}^t dv \, (A_{(v;y)}\theta_{\ge t})(v) \, G(s;y_0,v_0;y,v),
\end{equation}
whenever this integral is well defined.
\end{remark}

\begin{proof}
Formula \eqref{eq1thsubs} follows from \eqref{eq3lemsubs3} and  \eqref{eq1lemsubs}.
Next, by duality (effectively by the shift of the variable of integration),
\[
\int_t^{\infty} A^*_{(v;y)}G(s;y_0,v_0; y,v) \, dv
=\int_{-\infty}^{\infty} (\theta_{\ge t})(v) A^*_{(v;y)}G(s;y_0,v_0; y,v) \, dv
\]
\[
=\int_{-\infty} ^{\infty}  (A_{(v;y)}\theta_{\ge t})(v) G(s;y_0,v_0; y,v) \, dv
=\int_{v_0}^t dv \int_{t-v}^{\infty} \nu(v,w;y) \, dw \, G(s;y_0,v_0; y,v),
\]
implying \eqref{eq2thsubs}.
\end{proof}

\section{On the rate of convergence for the standard CTRW}
\label{secrates}

It is well known (see e.g. \cite{Meer09} and \cite{Ko11}) that the Markov chains with
jumps distributed like \eqref{eqpowertail}
converge after appropriate scaling to stable subordinators. In the following result we give
exact rates of weak convergence for the corresponding generators under  \eqref{eqpowertail}.

 \begin{prop}
 \label{apprrateCTRW}
 Let $p(y)$ be a probability density on $\R_+$ such that
 $p(y)=y^{-1-\al}$ for $y\ge B$ with some $\al \in (0,1)$ and $B>0$. Then,
 for any continuous $f$ on $\R_+$ such that $f(0)=f(\infty)=0$ and $f$ is Lipschitz
at zero so that $|f(y)| \le L y$ for $y\in [0,Bh]$ and some constant $L$, it follows that
 \begin{equation}
 \label{eq1apprrateCTRW}
 \left|h^{-\al} \int_0^{\infty} f(hy) p(y) dy-\int_0^{\infty} \frac{f(y) dy}{y^{1+\al}}\right| \le C_B L h^{1-\al} ,
\end{equation}
with
\[
C_B =\frac{B^{1-\al}}{1-\al} +\int_0^B y p(y)dy.
\]
 \end{prop}

\begin{remark}
If, instead of assuming Lipschitz at zero, one assumes H\"older continuity at zero,
that is $|f(y)| \le L y^{\be}$ for $y\in [0,Bh]$ and some $L>0$, $\be >\al$, then inequality
\eqref{eq1apprrateCTRW} holds with the r.h.s.
\[
\left(\frac{B^{\be-\al}}{\be-\al} +\int_0^B y^{\be} p(y) dy\right) L h^{\be-\al}.
\]
\end{remark}

\begin{proof}
Let $f=\1_{\ge C}$ with $C\ge Bh$. Then
\[
h^{-\al} \int_0^{\infty} f(hy) p(y) dy=h^{-\al} \int_{C/h}^{\infty} p(y) dy
\]
\[
=h^{-\al} \al^{-1} (C/h)^{-\al}=\al^{-1} C^{-\al}
=\int_C^{\infty} \frac{dy}{y^{1+\al}}=\int_0^{\infty} \frac{f(y) \, dy}{y^{1+\al}}.
 \]
Hence by linearity and the  density argument it follows that
\[
h^{-\al} \int_0^{\infty} f(hy) p(y) dy=\int_0^{\infty} \frac{f(y) dy}{y^{1+\al}},
\]
for any bounded measurable $f$ having support on $[Bh,\infty)$.
Next let $f$ have support on $[0,Bh]$. Then
\[
\left|h^{-\al} \int_0^{\infty} f(hy) p(dy)\right|=\left|h^{-\al} \int_0^B f(hy) p(y) dy\right|
\le h^{1-\al} L \int_0^B y p(y) dy,
\]
and
\[
\left|\int_0^{\infty} \frac{f(y) dy}{y^{1+\al}}\right|=\left|\int_0^{Bh} \frac{f(y) dy}{y^{1+\al}}\right|
\le L \int_0^{Bh} \frac{dy}{y^{\al}}=\frac{B^{1-\al}}{1-\al}Lh^{1-\al},
\]
implying \eqref{eq1apprrateCTRW}.
\end{proof}

In particular, setting $\tau=h^{\al}$, it follows that
 \begin{equation}
 \label{eq10apprrateCTRW}
 \left|\tau ^{-1} \int_0^{\infty} (f(x \pm \tau^{1/\al} y)-f(x)) p(y) dy
 -\int_0^{\infty} \frac{(f(x\pm y)-f(x)) dy}{y^{1+\al}}\right| \le C_B L \tau^{(1/\al)-1},
\end{equation}
where $L$ is the sup of the derivative of $f$ near $x$.

\section{Proof of Theorem \ref{th1}}
\label{secpr1}

(i)  We have
\[
\frac{1}{\tau}(U^{\tau}F-F)(x,s)
\]
\[
= \frac{1}{\tau}\int_{\R_+}\int_{\R^d} [F(x+\tau^{1/\be} y,  s+\tau^{1/(\al a(s,x))}r)-F(x,s)] Q_{s,x}(r) \, dr \, p(x,dy).
\]
 \begin{equation}
\label{scalingMarlift}
=\frac{1}{\tau}\int_{\R_+} [F(x,  s+\tau^{1/(\al a(s,x))}r)-F(x,s)] Q_{s,x}(r) \, dr
+\frac{1}{\tau}\int_{\R^d} [F(x+\tau^{1/\be} y,  s)-F(x,s)]  p(x,dy)+R,
\end{equation}
where
\[
R=\frac{1}{\tau}\int_{\R_+} \int_{\R^d} [g_y(x,  s+\tau^{1/(\al a(s,x))}r)-g_y(x,s)] Q_{s,x}(r) \, dr \, p(x,dy),
\]
with
\[
g_{y}(x,s)=F(x+\tau^{1/\be} y,s)-F(x,s).
\]

By \eqref{eq10apprrateCTRW}, the first term in   \eqref{scalingMarlift} converges, as $\tau\to 0$, to
\[
\int_0^{\infty} \frac {F(x, s+r)-F(x,s)}{r^{1+\al a(s,x)}} dr,
\]
whenever $F$ is continuously differentiable in $s$.
By \eqref{eqstaoperap1}, the second term in
\eqref{scalingMarlift} converges, as $\tau\to 0$, to $L_{st}^{\be}F(x,s)$.
Again by  \eqref{eq10apprrateCTRW} and \eqref{eqstaoperap1} it follows that
$R \to 0$, as $\tau \to 0$, implying \eqref{eq1th1}.

(ii) The operator $L$ is the sum of two terms, each one of them generates a Feller semigroup
on $C_{\infty}(\R^{d+1})$ such that the spaces $C^2_{\infty}(\R^{d+1})$ and $C^4_{\infty}(\R^{d+1})$
represent its invariant core. The fact that $L$ generates a Feller semigroup on $C_{\infty}(\R^{d+1})$
with an invariant core $C^2_{\infty}(\R^{d+1})$ follows from Theorems 5.2.1 and  5.2.2 of \cite{Kobook19}.
On the other hand, there is a well known technique of building transition probabilities for stable-like
processes with the generators of type \eqref{eq1th1}. Though the author is unaware of any result for
generator \eqref{eq1th1} exact, technique of \cite{Ko00}  (see also Chapter 7 of \cite{Ko11}) is fully
applicable yielding the existence of 4 times differentiable transition probabilities  $G(t;x,s; y,v)$, $t>0$,
for the process $(X_{x,s}, S_{x,s})(t)$.  This provides also
an alternative proof of the statement about the cores.

(iii) It is a direct consequence of (i),(ii) and the standard general result on the convergence
of Markov chains, see e.g. Theorem 19.28 of \cite{Kal}.

\section{Subordination for discrete Markov chains}
\label{secsubdis}

As a starting point for the proof of Theorem \ref{th2},
we give here some preliminary calculations for the subordinated Markov chains in discrete times.

Let $Y_{k\tau}$ be an adapted process on a stochastic basis $(\Om, \FC, \FC_t, P)$ and $\Si$
a stopping time with values in $\{k\tau\}$, $k\in \N$.
Let the random variables  $Y_{m\tau} \1(\Si=k\tau))$ have densities $g_{m\tau}(y, k\tau)$.
Then, for any $K>0$ and a continuous bounded function $F(y)$,
\begin{equation}
 \label{eq1lemsubsdis}
\E [F(Y_{\Si})\1(\Si\in [1/K,K])]= \sum_{k \tau \in [1/K,K]} \int \, g_{k\tau}(y, k\tau) F(y) dy.
\end{equation}

Let $(Y_{m\tau}, V_{m\tau})$, $Y\in \R^d$, $V \in \R$, be a Markov chain with the
transition operator $U^{\tau}$:
\[
U^{\tau}F(y_0,v_0)=\int F(y,v) P_{\tau} (y_0,v_0;y,v) \, dy dv=\int F(y,v) P_{\tau} (y_0,v_0;dy,dv),
\]
such that $Y_{m\tau}$ is a Markov chain on its own with the transition $U^{\tau}_Y$.

Then the density of the random vector $(Y_{k\tau}, V_{k\tau}, V_{(k-1)\tau})$ is
\begin{equation}
 \label{eqjointdis}
\int P_{\tau}(p,v; y,w)P_{(k-1)\tau}(y_0,v_0; p,v) \, dp,
\end{equation}
because
\[
\E F (Y_{k\tau}, V_{k\tau}, V_{(k-1)\tau})
\]
\[
=\int \E (F(Y_{k\tau}, V_{k\tau}, v)|Y_{(k-1)\tau}=p, V_{(k-1)\tau}=v)P_{(k-1)\tau}(y_0,v_0; p,v) \, dp dv
\]
\[
=\int F(y, w, v)P_{\tau}(p,v; y,w)P_{(k-1)\tau}(y_0,v_0; p,v) \, dp dv dw dy.
\]

Let the coordinate $V_{k\tau}$ be strictly increasing, and let
\[
Z_t^{\tau}=\sup\{ m\tau : V(m\tau)\le t\}=\inf\{ m\tau: V(m\tau)> t\}, \quad t>v_0,
\]
be its (generalized) inverse (or hitting times) process.

Then
\[
\E [F(Y_{Z_t^{\tau}})\1(Z_t^{\tau}\in [1/K,K])]
=\sum_{k \tau \in  [1/K,K]} F(Y_{k\tau}) \1(Z^{\tau}_t=k\tau)
=\sum_{k \tau \in  [1/K,K]} F(Y_{k\tau}) \1(V_{(k-1)\tau}<t\le V_{k\tau})
\]
\begin{equation}
 \label{eq4lemsubsdis}
 = \int F(y) \sum_{k \tau \in  [1/K,K]} \int \1(v<t\le w)
  P_{\tau}(p,v; y,w)P_{(k-1)\tau}(y_0,v_0; p,v) \, dy dp dv dw.
\end{equation}

This can be rewritten as
\begin{equation}
 \label{eq5lemsubsdis}
\E [F(Y_{Z_t^{\tau}})\1(Z_t^{\tau}\in [1/K,K])]
 = \int \sum_{k \tau \in  [1/K,K]} (U^{\tau}F_v)(p,v) P_{(k-1)\tau}(y_0,v_0; p,v) \, dp dv,
\end{equation}
with $F_v(y,w)= F(y)  \1(v<t\le w)$.

\section{Proof of Theorem \ref{th2}}
\label{secpr2}

Theorem \ref{thsubs} implies formula
\eqref{eq20th2}:
\[
\E [F(\tilde X_{x,s}(t))\1(T_{x,s}(t)\in [1/K,K])]
\]
\begin{equation}
\label{eq20th2rep}
=\int dy \int_{1/K}^K du \int_s^t dv \frac{(t-v)^{-a(v,y)\al}}{a(v,y)\al} \, G(u;x,s; y,v) F(y).
\end{equation}

To complete the proof of statement (ii) of Theorem \ref{th2} one has to note that, due to the
asymptotic behavior of the transition probabilities  $G(u;x,s; y,v)$ of the stable-like process
$ (X_{x,s}, S_{x,s})(t)$ (see \cite{Ko00} or \cite{Ko11}), the function $G(u;x,s; y,v)$ is
integrable for small and large $u$, so that one can pass to the
limit $K\to \infty$ in \eqref{eq20th2rep} to obtain  \eqref{eq20th2rep}.

Next, by \eqref{eq5lemsubsdis},
\begin{equation}
 \label{eq5lemsubsdisrep}
\E [F(\tilde X^{\tau}_{x,s}(t))\1(T^{\tau}_{x,s}(t)\in [1/K,K])]
 = \int dy \int_0^t dv \, \sum_{k \tau \in  [1/K,K]} (U^{\tau}F_v)(y,v) P_{(k-1)\tau}(x,s; y,v),
\end{equation}
where $F_v(q,w)= F(q)  \1(v<t\le w)$ and
\begin{equation}
\label{eqtransitionenhrep}
U^{\tau}F(y,v)= \int_{\R_+}\int_{\R^d} F(y+\tau^{1/\be} q,  v+\tau^{1/(\al  a(v,y))} r) Q_{v,y}(r) \, dr \, p(y,dq),
\end{equation}
with $Q_{s,x}$ satisfying \eqref{eqpowertail1}.

Therefore,
\[
\E [F(\tilde X^{\tau}_{x,s}(t))\1(T^{\tau}_{x,s}(t)\in [1/K,K])]
\]
\begin{equation}
\label{eq1th2proof}
 = \int dy \int_0^t dv \,  \sum_{k\tau \in  [1/K,K]}
 P_{(k-1)\tau}(x,s; y,v) F(y+ \tau^{1/\be} q) p(y,dq)  \1(v<t\le v+\tau^{1/(\al  a(v,y))} r) Q_{v,y}(r) \, dr.
\end{equation}

Since the probability of small and large jumps are small both for $S_{x,s}$ and $S^{\tau}_{x,s}$,
 it follows that, in
order to prove Statement (i) of Theorem \ref{th2}, it is sufficient to show that
expression \eqref{eq1th2proof} tends to expression \eqref{eq20th2rep}, as $\tau \to 0$, for any finite $K$
and any $F\in C_{\infty} (\R^d)$. By the density argument this can be further reduced to
the case of $F\in C^2_{\infty} (\R^d)$.

The key point for the argument is that, by Theorem \ref{th1}, the sums
\[
 \tau \int dy \int_0^t dv \,  \sum_{k\tau \in  [1/K,K]}
 P_{(k-1)\tau}(x,s; y,v) \Om(y,v)
 \]
approximate uniformly the Riemannian sums for the integral
\[
\int dy \int_0^t dv \,  \int_{[1/K,K]} du \,
 G(u; x,s; y,v) \Om(y,v),
 \]
 and thus converge to this integral, at least for $\Om \in C_{\infty}(\R^{d+1})$.

Turning concretely to \eqref{eq1th2proof}, we note firstly that the limit must be the same for the expression

\begin{equation}
 \label{eq2th2proof}
\int dy \int_0^t dv \,  \sum_{k \tau\in  [1/K,K]}
 P_{(k-1)\tau}(x,s; y,v) F(y)  \1(v<t\le v+\tau^{1/(\al  a(v,y))} r) Q_{v,y}(r) \, dr.
\end{equation}

In fact, expressions \eqref{eq1th2proof} and \eqref{eq2th2proof} differ by
$\int (F(y+\tau^{1/\be} q)-F(y)) p(y,dq)$, which tends to zero, as $\tau \to 0$, uniformly in $y$
(for $F\in  C^2_{\infty} (\R^d)$).
Hence, if functions \eqref{eq2th2proof} have a finite limit (as we are going to show),
then the same limit have functions \eqref{eq1th2proof}.

Since
 \begin{equation}
 \label{eq3th2proof}
\frac{1}{\tau}\int_0^{\infty}  \1(v<t\le v+\tau^{1/(\al  a(v,y))} r) Q_{v,y}(r) \, dr
=\frac{1}{\tau} \int_0^{\infty}  [\theta_{\ge t}(v+\tau^{1/a(v,y)\al}r)-\theta_{\ge t}(v)] Q_{v,y}(r) \, dr,
\end{equation}
it tends to $A \theta_{\ge t}(v)$, by \eqref{eq10apprrateCTRW}, as claimed. Therefore, we just need
to show that this limit goes through the Riemannian sum approximation \eqref{eq2th2proof}. For this end we note
that the limit of \eqref{eq2th2proof} remains the same if we slightly decrease the domain of integration, namely
if we choose the expression
\begin{equation}
 \label{eq4th2proof}
\int dy \int_{[0, t-B\tau^{1/a(v,y) \al}]} dv \,  \sum_{k \tau\in  [1/K,K]}
 P_{(k-1)\tau}(x,s; y,v) F(y)  \1(v<t\le v+\tau^{1/(\al  a(v,y))} r) Q_{v,y}(r) \, dr.
\end{equation}

In fact, by \eqref{eq3th2proof} and \eqref{eqpowertail1}, the difference between
\eqref{eq2th2proof} and  \eqref{eq4th2proof}  is bounded in magnitude by the expression
\[
\tau \int dy \int_{[t-B\tau^{1/a(v,y) \al},t]} dv \,  \sum_{k \tau\in  [1/K,K]}
 P_{(k-1)\tau}(x,s; y,v) F(y) \Om_{\tau}(y,v),
\]
with $\Om(y,v)\le 1/\tau$. And this expression tends to zero, as $\tau \to 0$, because
\[
\int_{[t-B\tau^{1/a(v,p) \al},t]}  \Om_{\tau}(p,v) \, dv
\]
is bounded by $B\tau^{(1/a(v,y) \al)-1}$, which is uniformly small.

Next, for $v>B\tau^{1/a(v,y) \al}$,
\[
\frac{1}{\tau}\int_0^{\infty}  \1(v<t\le v+\tau^{1/(\al  a(v,y))} r) Q_{v,y}(r) \, dr
=A \theta_{\ge t}(v).
\]
(see proof of Proposition \ref{apprrateCTRW} for this simple calculation).

Hence it remains to find the limit of the expression
\begin{equation}
 \label{eq5th2proof}
\tau \int dy \int_{[0, t-A\tau^{1/a(v,y) \al}]} dv \,  \sum_{k \tau\in  [1/K,K]}
 P_{(k-1)\tau}(x,s; y,v) F(y) A \theta_{\ge t}(v).
\end{equation}

Since $G(u; x,s; y,v)$ is continuous in $y,v$ and  $A \theta_{\ge t}(v)$ has an integrable singularity,
it follows that this expression tends to \eqref{eq20th2rep} completing the proof of part (i) of Theorem \ref{th2}.

\section{Proof of Theorem \ref{th3}}
\label{secpr3}

The statement on the probabilistic representation of the solutions to \eqref{eq2th2} follows from the
standard argument (see e.g. \cite{Ko19}) with Dynkin's martingale.
Namely, for a function $F$ from the domain of the generator $\La$ of the process
$(X_{x,s},S_{x,s})(t)$, the process
\[
M^t_F=F(X_{x,s}(t),S_{x,s}(t))-F(x,s)-\int_0^t \La F (X_{x,s},S_{x,s})(r) \, dr
\]
is a martingale, called  Dynkin's martingale. We are now looking at the modification of
the process $(X_{x,s},S_{x,s})(t)$ obtained by stopping it at the attempt to cross the boundary $S_{x,s}=t$.
This modification specifies a process on $\R^d \times (-\infty,t]$ with the generator
\[
\tilde \La (x,s)=L_{st}^{\be}F(x,s)
\]
\begin{equation}
\label{eq6th2proof}
+\int_0^{t-s} \frac {F(x, s+r)-F(x,s)}{r^{1+\al a(s,x)}} dr
+(F(x,t)-F(x,s))\frac{(t-s)^{-a(s,x) \al}}{a(s,x) \al},
\end{equation}
obtained formally from \eqref{eq1th1} by reducing it to
the set of functions such that $F(x,r)=F(x,t)$ for all $r\ge t$.

Applying Doob's optional sampling theorem to the martingale $M^t_F$ and the stopping time
\[
T_{x,s}(t)=\inf \{r: S_{x,s} (r) \ge t\},
\]
we can conclude that, if $\tilde \La F=0$, which is equivalent to equation \eqref{eq2th2}, then
\begin{equation}
 \label{eq7th2proof}
F(x,s)=\E F[\tilde X_{x,s}(t)] =\E F[ X_{x,s} (T_{x,s}(t))].
\end{equation}

Thus formula \eqref{eq7th2proof} gives the probabilistic representation
for a solution of the fractional boundary problem \eqref{eq2th2} implying also the
uniqueness of such a solution. To complete the proof of well-posedness it remains to
check that formula \eqref{eq7th2proof} really solves the problem.
The fact that it satisfies the boundary condition  $F(x,t)=F(x)$
is a direct consequence of \eqref{eq7th2proof}. Let us check that it satisfies the equation
$\tilde \La F=0$. To this end, we shall use the representation \eqref{eq21th2}:
\[
F(x,s)=\E [F(Y_{Z_t})]
= \int dy \int_0^{\infty} du \int_0^t dv \, (A_{(v;y)}\theta_{\ge t})(v)G(u;x,s; y,v) F(y).
\]

We know that $F(x,t)=F(x)=\lim_{s\to t} F(x,s)$.
Setting $F(x,s)=F(x)$ for $s\ge t$, we make the function $F(x,s)$ continuous everywhere.

We want to show that
\begin{equation}
 \label{eq8th2proof}
L_{st}^{\be} F(x,s)+A_{(s;x)}F(s,x)=0.
\end{equation}

It is straightforward from \eqref{eq21th2} that $F(x,s)$ is smooth for $s<t$,
so that all terms in this equation are well defined. Let
\[
f(x,s;u,v)= \int dy \, G(u;x,s; y,v) F(y) (A_{(v;y)}\theta_{\ge t})(v), \quad u>0,
\]
so that $f(x,s;u,v)=0$ for $s\ge v$ and any $u>0$.
By the Kolmogorov's equation for the transition density,
\[
\frac{\pa G}{\pa u}= L_{st}^{\be} G+A_{(s;x)}G.
\]
Therefore
\[
\frac{\pa f}{\pa u}= L_{st}^{\be} f+A_{(s;x)}f.
\]
Let
\[
g(x,s; v)=\int_0^{\infty} f(x,s;u,v) \, du \,\, \text{for} \,\, s< v \,\, \text{and} \,\, g(x,s;v)=0 \,\, \text{for} \,\, s> v.
\]
Integrating the equation for $f$ over $u$ and using that $G(0;x,s;y,v)=\de(x-y)\de(s-v)$ yields the equation
\[
 (L_{st}^{\be}+A_{(s;x)})g(x,s; v)+ F(x) (A_{(v;x)}\theta_{\ge t})(v) \de (v-s)=0.
 \]

 Since
 \[
 F(x,s)=\int_s^t  g(x,s; v) dv, \quad s\le t \,\, \text{and} \,\, F(x,s)=F(x) \,\, \text{for} \,\, s\ge t,
 \]
 it follows that
 \[
 L_{st}^{\be} F(x,s)+\int_s^t A_{(s;x)}g(x,s; v) dv +F(x) (A_{(s;x)}\theta_{\ge t})(s)=0.
 \]

Here
\[
\int_s^t A_{(s;x)}g(x,s; v) dv+F(x) (A_{(s;x)}\theta_{\ge t})(s)
\]
\[
=\int_s^t dv \int_0^{t-s}  dr (g(x,s+r; v)-g(x,s;v))r^{-1-a(s,x) \al}
+(A_{(s;x)}\theta_{\ge t})(s) (F(x)-F(x,s)).
\]

Consequently, in order to establish \eqref{eq8th2proof}, it remains to show that
\begin{equation}
\label{eqproofsol}
 \int_s^t dv \int_0^{t-s}  dr (g(x,s+r; v)-g(x,s;v))r^{-1-a(s,x) \al}
 = \int_0^{t-s}  dr (F(x,s+r)-F(x,s))r^{-1-a(s,x) \al}.
\end{equation}

Let $\chi_{\ep}(s,r)$ be positive functions that coincide with $r^{-1-a(s,x) \al}$ for $r\le \ep$,
do not exceed these functions everywhere, and are continuous and bounded. Then to prove
\eqref{eqproofsol} it is sufficient to show
\begin{equation}
\label{eqproofsol1}
 \int_s^t dv \int_0^{t-s}  dr (g(x,s+r; v)-g(x,s;v))\chi_{\ep}(s,r)
 = \int_0^{t-s}  dr (F(x,s+r)-F(x,s))\chi_{\ep}(s,r),
\end{equation}
for any $\ep$.
The second terms cancel and this equation turns to
\begin{equation}
\label{eqproofsol2}
 \int_s^t dv \int_0^{t-s}  dr g(x,s+r; v)\chi_{\ep}(s,r)
 = \int_0^{t-s}  dr F(x,s+r)\chi_{\ep}(s,r).
\end{equation}
And this holds, because the l.h.s. is
\[
\int_0^{t-s}  dr  \int_{s+r}^t dv \, g(x,s+r; v)\chi_{\ep}(s,r)
=\int_0^{t-s}  dr \, F(x,s+r)\chi_{\ep}(s,r).
\]

%\begin{remark} A pure analytic proof of Theorem \ref{th3} can be constructed
%by rewriting fractional differential equation in the integral form, as was done
%for many similar equations in \cite{Kobook19}.
%\end{remark}

%\section{Discussion}
%\label{secdis}

{\bf Acknowledgements}. 

The author is grateful to the organisers of the 50th Memorial Barrett Lectures in the
University of Tennessee, May 2021, for inviting him to give a plenary talk, which strongly stimulated
his further research in fractional calculus.    

The author is grateful to the Simons foundation for the support of his residence at INI in Cambridge
during the program Fractional Differential Equations Jan-Apr 2022.

The work was supported by the Russian Science Foundation
under grant no. 20-11-20119.


\begin{thebibliography}{99}

\bibitem{Scalabook}
D. Baleanu, K. Diethelm, E. Scalas and J.J.  Trujillo. Fractional calculus: Models and numerical methods: Second edition.
World Scientific Publishing, Singapore, 2017. Series on Complexity, Nonlinearity and Chaos, {\bf 5}.

\bibitem{Ding21}
W. Ding et al. Applications of distributed-order fractional operators: a review.
Entropy 23 (2021), no. 1, Paper No. 110.

\bibitem{Fed}
S. Fedotov, D. Han, A Zubarev, M. Johnston and V. J. Allan.
Variable-order fractional master equation and clustering of particles: non-uniform lysosome distribution.
Philos. Trans. Roy. Soc. A 379 (2021), no. 2205, Paper No. 20200317.

%\bibitem{GnedKor}
%B. V. Gnedenko and V. Yu. Korolev, Random Summation: Limit Theorems and Applications, CRC Press, Boca
%Raton, Florida, 1996.

\bibitem{Kal} O. Kallenberg. Foundations of Modern Probability.
Second ed., Springer 2002.

\bibitem{Kir94} V. Kiryakova, {\it Generalized fractional calculus and applications,}
(Longman Scientific, Harlow. Copublished in the United States with John Wiley and Sons, New York, 1994.
 Pitman Research Notes in Mathematics Series, {\bf 301}.

%\bibitem{KochKondr17} A.~N.~Kochubei and Y.~Kondratiev, {``Fractional kinetic hierarchies and intermittency,''}
%Kinet. Relat. Models {\bf 10} (3), 725--740 (2017).

\bibitem{Ko00}
V.N. Kolokoltsov. Symmetric
Stable Laws and Stable-like Jump-Diffusions. Proc. London Math.
Soc. {\bf 3 (80)} (2000), 725-768.

\bibitem{Ko09} V. N. Kolokoltsov. Generalized Continuous-Time Random Walks (CTRW),
 Subordination by Hitting Times and Fractional Dynamics.
 Theory Probab. Appl. {\bf 53} (4), 594--609 (2009).


\bibitem{Ko11}
V. N. Kolokoltsov. {\sl Markov processes, semigroups and generators}.
DeGruyter Studies in Mathematics v. 38, DeGruyter, 2011.


%\bibitem{Ko15} V. N. Kolokoltsov,
%On fully mixed and multidimensional extensions of the Caputo and Riemann-Liouville derivatives,
% related Markov processes and fractional differential equations.
% Fract. Calc. Appl. Anal. {\bf 18} (4), 1039--1073 (2015).

\bibitem{Ko19}
V. N.  Kolokoltsov.  The probabilistic point of view on the generalized fractional PDEs. FCAA 22:3 (2019),
543-600 (open access mode).

\bibitem{Kobook19} V. N. Kolokoltsov,
{\it Differential equations on measures and functional spaces},
(Birkh\"auser, Cham, 2019) Birkh\"auser Advanced Texts Basler Lehrb\"ucher.

\bibitem{KolMal}
 V. N. Kolokoltsov and O. A. Malafeyev.
 Many Agent Games in Socio-economic Systems:
Corruption, Inspection, Coalition Building, Network Growth, Security.
Springer Series in Operations Research and Financial Engineering, Springer Nature, 2019.

\bibitem{KoTro}
V. Kolokoltsov and M. Troeva.
A New Approach to Fractional Kinetic Evolutions.
Fractal Fract. 2022, 6, 49.
https://doi.org/10.3390/fractalfract6020049

\bibitem{Meer09}
 M. M. Meerschaert, E. Nane and Y. Xiao.
 Correlated continuous time random walks. Statist. Probab. Lett. 79:9 (2009), 1194 - 1202.

\bibitem{Podlub99}
I. Podlubny,
{\it Fractional differential equations, An introduction to fractional derivatives,
fractional differential equations, to methods of their solution and some of their applications}
(Academic Press, Inc., San Diego, 1999), Mathematics in Science and Engineering {\bf 198}.

\bibitem{SavToa}
M. Savov and B. Toaldo. Semi-Markov processes, integro-differential equations and anomalous diffusion-aggregation.
Ann. Inst. Henri Poincar\'e Probab. Stat. 56:4 (2020), 2640 - 2671.

\bibitem{Straka}
Peter Straka.
Variable Order Fractional Fokker-Planck Equations derived from
Continuous Time Random Walks. Phys. A 503 (2018), 451 - 463.

\bibitem{Sun}
H. Sun et al. A review on variable-order fractional differential equations: Mathematical foundations,
physical models, numerical methods and applications. Fract. Calc. Appl. Anal. 2019, 22, 27 - 59.

\bibitem{Zas94}
G.M. Zaslavsky. Fractional kinetic equation for Hamiltonian chaos. Physica D 76
(1994), 110-122.

\end{thebibliography}
\end{document}